\documentclass[a4paper, reqno, 14pt]{amsart}

\usepackage[usenames,dvipsnames]{color}

\usepackage{amsthm,amsfonts,amssymb,amsmath,amsxtra}
\usepackage[all]{xy}
\SelectTips{cm}{}
\usepackage{xr-hyper}
\usepackage[colorlinks=
   citecolor=Black,
   linkcolor=Red,
   urlcolor=Blue]{hyperref}
\usepackage{verbatim}

\usepackage[margin=1.25in]{geometry}

\usepackage{mathrsfs}

\RequirePackage{xspace}
\RequirePackage{etoolbox}
\RequirePackage{varwidth}
\RequirePackage{enumitem}
\RequirePackage{tensor}
\RequirePackage{mathtools}
\RequirePackage{longtable}
\RequirePackage{multirow}

\newcommand{\bE}{\mathbf E}
\newcommand{\bG}{\mathbf G}
\newcommand{\bK}{\mathbf K}
\newcommand{\bM}{\mathbf M}

\newcommand{\BA}{\ensuremath{\mathbb {A}}\xspace}

\newcommand{\BF}{\ensuremath{\mathbb {F}}\xspace}
\newcommand{{\BG}}{\ensuremath{\mathbb {G}}\xspace}

\newcommand{{\BK}}{\ensuremath{\mathbb {K}}\xspace}

\newcommand{\BN}{\ensuremath{\mathbb {N}}\xspace}

\newcommand{\BQ}{\ensuremath{\mathbb {Q}}\xspace}

\newcommand{\BZ}{\ensuremath{\mathbb {Z}}\xspace}

\newcommand{\CG}{\ensuremath{\mathcal {G}}\xspace}

\newcommand{\CM}{\ensuremath{\mathcal {M}}\xspace}

\newcommand{\CO}{\ensuremath{\mathcal {O}}\xspace}

\newcommand{\Ad}{{\mathrm{Ad}}}

\DeclareMathOperator{\Aut}{Aut}

\DeclareMathOperator{\Adm}{Adm}

\DeclareMathOperator{\sstr}{\sigma-{\rm str}}

\DeclareMathOperator{\Gal}{Gal}
\newcommand{\GL}{\mathrm{GL}}

\newcommand{\id}{\ensuremath{\mathrm{id}}\xspace}
\let\Im\relax
\DeclareMathOperator{\Im}{Im}

\newcommand{\inv}{{\mathrm{inv}}}

\DeclareMathOperator{\Spec}{Spec}

\newcommand{\wt}{\widetilde}
\newcommand{\wh}{\widehat}

\newcommand{\ov}{\overline}

\newtheorem{theorem}{Theorem}
\newtheorem{proposition}[theorem]{Proposition}
\newtheorem{lemma}[theorem]{Lemma}

\newtheorem{corollary}[theorem]{Corollary}
\newtheorem{axiom}[theorem]{Axiom}
\theoremstyle{definition}
\newtheorem{definition}[theorem]{Definition}

\newtheorem{remark}[theorem]{Remark}
\newtheorem{remarks}[theorem]{Remarks}

\newenvironment{altenumerate}
   {\begin{list}
      {\textup{(\theenumi)} }
      {\usecounter{enumi}
       \setlength{\labelwidth}{0pt}
       \setlength{\labelsep}{0pt}
       \setlength{\leftmargin}{0pt}
       \setlength{\itemsep}{\the\smallskipamount}
       \renewcommand{\theenumi}{\roman{enumi}}
      }}
   {\end{list}}
\newenvironment{altitemize}
   {\begin{list}
      {$\bullet$}
      {\setlength{\labelwidth}{0pt}
	   \setlength{\itemindent}{5pt}
       \setlength{\labelsep}{5pt}
       \setlength{\leftmargin}{0pt}
       \setlength{\itemsep}{\the\smallskipamount}
      }}
   {\end{list}}

\numberwithin{equation}{section}
\numberwithin{theorem}{section}

\setitemize[0]{leftmargin=*,itemsep=\the\smallskipamount}
\setenumerate[0]{leftmargin=*,itemsep=\the\smallskipamount}

\renewcommand{\to}{%
   \ifbool{@display}{\longrightarrow}{\rightarrow}%
   }
\let\shortmapsto\mapsto
\renewcommand{\mapsto}{%
   \ifbool{@display}{\longmapsto}{\shortmapsto}%
   }
\newlength{\olen}
\newlength{\ulen}
\newlength{\xlen}
\newcommand{\xra}[2][]{%
   \ifbool{@display}%
      {\settowidth{\olen}{$\overset{#2}{\longrightarrow}$}%
       \settowidth{\ulen}{$\underset{#1}{\longrightarrow}$}%
       \settowidth{\xlen}{$\xrightarrow[#1]{#2}$}%
       \ifdimgreater{\olen}{\xlen}%
          {\underset{#1}{\overset{#2}{\longrightarrow}}}%
          {\ifdimgreater{\ulen}{\xlen}%
             {\underset{#1}{\overset{#2}{\longrightarrow}}}
             {\xrightarrow[#1]{#2}}}}%
      {\xrightarrow[#1]{#2}}
   }
\makeatother
\newcommand{\xyra}[2][]{%
   \settowidth{\xlen}{$\xrightarrow[#1]{#2}$}%
   \ifbool{@display}%
      {\settowidth{\olen}{$\overset{#2}{\longrightarrow}$}%
       \settowidth{\ulen}{$\underset{#1}{\longrightarrow}$}%
       \ifdimgreater{\olen}{\xlen}%
          {\mathrel{\xymatrix@M=.12ex@C=3.2ex{\ar[r]^-{#2}_-{#1} &}}}%
          {\ifdimgreater{\ulen}{\xlen}%
             {\mathrel{\xymatrix@M=.12ex@C=3.2ex{\ar[r]^-{#2}_-{#1} &}}}
             {\mathrel{\xymatrix@M=.12ex@C=\the\xlen{\ar[r]^-{#2}_-{#1} &}}}}}%
      {\mathrel{\xymatrix@M=.12ex@C=\the\xlen{\ar[r]^-{#2}_-{#1} &}}}%
   }
\makeatletter
\newcommand{\xla}[2][]{%
   \ifbool{@display}%
      {\settowidth{\olen}{$\overset{#2}{\longleftarrow}$}%
       \settowidth{\ulen}{$\underset{#1}{\longleftarrow}$}%
       \settowidth{\xlen}{$\xleftarrow[#1]{#2}$}%
       \ifdimgreater{\olen}{\xlen}%
          {\underset{#1}{\overset{#2}{\longleftarrow}}}%
          {\ifdimgreater{\ulen}{\xlen}%
             {\underset{#1}{\overset{#2}{\longleftarrow}}}
             {\xleftarrow[#1]{#2}}}}%
      {\xleftarrow[#1]{#2}}
   }
\newcommand{\isoarrow}{%
   \ifbool{@display}{\overset{\sim}{\longrightarrow}}{\xrightarrow\sim}%
   }

\begin{document}

\title[Stratifications]{Stratifications in the reduction of Shimura varieties}
\author{X. He}
\address{University of Maryland, Department of Mathematics, College Park, MD 20742-4015 U.S.A., and Department of Mathematics, HKUST, Hong Kong }
\email{xuhuahe@math.umd.edu}
\author{M. Rapoport}
\address{Mathematisches Institut der Universit\"at Bonn, Endenicher Allee 60, 53115 Bonn, Germany}
\email{rapoport@math.uni-bonn.de}
\thanks{X. H. was partially supported by NSF DMS-1463852. M.R. was supported by the Deutsche Forschungsgemeinschaft through the grant SFB/TR 45.}

\date{\today}
\maketitle

\tableofcontents

\section{Introduction}

This paper is about characteristic subsets in the reduction modulo $p$ of a general Shimura variety with parahoric level structure. We are referring to the \emph{Newton stratification},  the \emph{Ekedahl-Oort stratification} and the \emph{Kottwitz-Rapoport stratification}. The classic work on the first two kinds of stratifications concerns the Siegel case with hyperspecial level structure which was studied by F.~Oort and others. There is also a lot of work on other Shimura varieties which are familiar moduli spaces of abelian varieties, comp. the references in \cite{V-W}. Here we are concerned with  the group-theoretic versions of these stratifications \cite{Ko1, RR, H, Vi}. To the stratifications above we add here the \emph{Ekedahl-Kottwitz-Oort-Rapoport stratification} which interpolates between  the Kottwitz-Rapoport stratification in the case of an Iwahori level structure and the  Ekedahl-Oort  stratification of Viehmann \cite{Vi} in the hyperspecial case.  

Concerning these stratifications, one can ask many interesting questions (when do they have the strong stratification property?, when are the strata equi-dimensional and  what is their dimension? what are local and global properties of the individual strata--are they smooth, are they (quasi-)affine?, and so on). There is a large body of literature on these topics, but here we are only concerned with the question of which  strata are non-empty and of the relation between these various stratifications. 
The background of these questions is addressed in the survey paper by Haines \cite{H} and the survey paper of the 
second author \cite{R:guide}, where early work on these problems is described. 

As far as the non-emptiness of Newton strata in the group-theoretic setting (in their natural index set) is concerned, the goal  is to prove the conjecture of  the second author \cite[Conj. 7.1]{R:guide}, comp. also Fargues  \cite[p.~55]{Fa} and \cite[Conj. 12.2]{H}. This question is considered by Viehmann/Wedhorn  \cite{V-W} in the PEL-case, for hyperspecial level structures.  Recent work of D.-U. Lee \cite{Lee}, M. Kisin \cite{Ki} and C.-F. Yu \cite{Yu2} addresses this problem for  Shimura varieties of PEL-type, and even of Hodge type, when the underlying group is unramified at $p$. Kisin  proves the non-emptiness of the {\it basic stratum} even when the unramifiedness assumption is dropped. The method used in all these works is based on the papers \cite{LR} and \cite{Kpoints}. There is also work by Kret \cite{Kr1, Kr2} for  Shimura varieties of PEL-type which, besides the papers  \cite{LR, Kpoints}, also uses the Arthur trace formula. 
 
 The non-emptiness of EO strata (in their natural index set) is proved  by Viehmann/Wedhorn \cite{V-W} in the PEL-case. Non-emptiness of KR strata is due in the Siegel case to Genestier and in  the {\it fake unitary case} to Haines (comp. \cite[Lemma 13.1]{H}). 
 
We also mention the papers by G\"ortz/Yu \cite{GY} and by Viehmann/Wedhorn \cite{V-W} on the relation between KR strata, resp. EO strata, and Newton strata, and the work of G\"ortz/Hoeve \cite{GHo} and of Hartwig \cite{Har} on the relation between EO strata and KR strata for parahoric level structures in the Siegel case.

\smallskip

  The purpose of the present paper is to understand how to define these characteristic subsets in the most general case and to predict their existence and their  properties. Our approach is axiomatic. 
We  formulate a series of axioms, and show that, if these axioms are satisfied, then the existence theorems follow. Here the novelty of our approach comes from the proof of one of us \cite{H2} of the Kottwitz-Rapoport conjecture from \cite{KoR1,R:guide}. In particular, our methods are purely group-theoretical and combinatorial and  use algebraic geometry only indirectly. Algebraic geometry would  become relevant when trying to check the axioms in a specific case.

We stress (if this is necessary at all!) that the aim of this paper is quite modest. We wanted to give a blueprint that could possibly be followed to achieve further progress on these questions, even for Shimura varieties that are not of PEL-type, or for the reduction modulo a prime number where the group defining the Shimura variety has bad reduction (i.e., is not unramified), or for the reduction modulo a prime above $2$. 

Our paper has an antecedent \cite{HW}. In this unpublished preprint of Wedhorn and the first of us, a similar circle of questions is addressed. But the point of view is quite different; in particular, in loc.~cit., the underlying group is supposed to be unramified at $p$. On the other hand, \cite{HW} contains results that are not superceded by the present  paper. 

The lay-out of the paper is as follows. In section \ref{s IWBG}, we fix our notation concerning Iwahori-Weyl groups and recall some facts about Kottwitz's set $B(G)$.  In section \ref{s:axioms}, we state the axioms on which our reasonings are based; we also state some auxiliary conjectures. In section \ref{s:KR} we show how to deduce from our axioms the non-emptiness theorems for KR strata. In section \ref{s: N}, we do the same for Newton strata.  In section \ref{s:EKOR}, we define Ekedahl-Kottwitz-Oort-Rapoport strata   and prove  non-emptiness theorems for them. In section \ref{s: Siegel}, we make  all our concepts explicit in the simplest case, the Siegel case. 

We thank U.~G\"ortz, P.~Scholze, T.~Wedhorn and C.-F.~Yu for helpful discussions. The first author also thanks T.~Wedhorn for the numerous discussions he had with him when \cite{HW} was conceived; the ideas developed in \cite{HW} definitely had an influence (conscious or unconscious) on the present paper.

\section{Recollections on  the Iwahori Weyl group and on $B(G)$}\label{s IWBG}
In this section, we collect some  facts concerning the group-theoretic background. Its main purpose is to introduce some notation used in the rest of the paper.

 Let $\breve \BQ_p$ be the completion of the maximal unramified extension of $\BQ_p$ in a fixed algebraic closure $\ov\BQ_p$, with ring of integers $O_{\breve \BQ_p}$. We denote by $\sigma$ its Frobenius automorphism.
 
 Let $G$ be a connected reductive algebraic group over $\BQ_p$.  We denote by $I$ an Iwahori subgroup of $G$. Since all Iwahori subgroups are conjugate, there will be no harm in only considering parahoric subgroups $K$ which contain $I$. If $K$ is such a parahoric subgroup, it defines a smooth group scheme $\CG=\CG_K$ over $\Spec \BZ_p$. We denote by $\breve K$ the subgroup $\breve K=\CG(O_{\breve \BQ_p})$ of $G(\breve \BQ_p)$. 

We fix a maximal torus $T$ which after extension of scalars is contained in a Borel subgroup of $G\otimes_{\BQ_p}\breve\BQ_p$, and such that $\breve I$ is the Iwahori subgroup fixing an alcove in the apartment attached to the split part of $T$. Denote by $N$ the normalizer of $T$. Then the {\it Iwahori Weyl group} is defined by $\tilde W=N(\breve\BQ_p)/(T(\breve\BQ_p)\cap \breve I)$, cf. \cite{HR}.  Let $W_0=N(\breve\BQ_p)/T(\breve\BQ_p)$. Then $\tilde W$ is a split extension of $W_0$ by the central  subgroup $X_*(T)_{\Gamma_0}$, with its natural $W_0$-action. Here $\Gamma_0=\Gal(\ov\BQ_p/\breve \BQ_p)$. The splitting depends on the choice of a special vertex of the base alcove that we fix in the sequel. When considering an element $\mu\in X_*(T)_{\Gamma_0}$ as an element of $\tilde W$, we write $t^\mu$. Recall the {\it $\{\mu\}$-admissible set}, associated to a conjugacy class of cocharacters of $G$,
 \begin{equation}
\Adm(\{\mu\})=\{w \in \tilde W; w \le t^{x(\underline \mu)} \text{ for some }x \in W_0\} .
\end{equation}
 Here $\underline \mu$ is the image in $X_*(T)_{\Gamma_0} $ of a dominant representative $\mu$ of the conjugacy class $\{\mu\}$, i.e., the one which lies in the Weyl chamber {\it opposite} to the unique Weyl chamber containing  the base alcove with apex at the fixed special vertex (it corresponds to a $\breve \BQ_p$-rational Borel subgroup containing $T$), cf. \cite[Rem.~4.~17]{PRS}. Also, we used the Bruhat order on $\tilde W$ defined by the choice of $I$. The definition of the Bruhat order  uses that $\tilde W$ is a split extension of $\pi_1(G)_{\Gamma_0}$ by the {\it affine Weyl group} which is a Coxeter group. 
 
Let $K$ be a parahoric subgroup containing $I$. Let 
$$W_K=\tilde W\cap \breve K=\big(N(\breve\BQ_p)\cap \breve K\big)/(T(\breve\BQ_p)\cap \breve I) .$$
   Then the Bruhat order on $\tilde W$ induces a partial order on the double coset space $
W_K \backslash \tilde W/W_K .$
Let $^K\tilde W^K$ be the set of minimal elements in their double coset by $W_K$. Then for $w, w'\in ^K\tilde W^K$, we have $W_K w W_K\leq W_K w W_K$ if and only if $w\leq w'$. 
We also set 
  \begin{equation}\label{defadm}
  \begin{aligned}
  \Adm(\{\mu\})^K=&W_K \Adm(\{\mu\}) W_K,& \quad \text{ a subset of $ \tilde W,$} \\
    \Adm(\{\mu\})_K=&W_K \backslash \Adm(\{\mu\})^K /W_K,& \quad \text{ a subset of $W_K \backslash \tilde W/W_K .$}
  \end{aligned}
  \end{equation} 
  Another object in the Iwahori Weyl group associated to the conjugacy class $\{\mu\}$ is 
  \begin{equation}\label{def tau}
  \tau_{\{\mu\}} ,
\end{equation}
  the unique element of $\tilde W$ of length zero mapping to the element $\mu^\sharp\in \pi_1(G)_{\Gamma_0}$, cf. \cite[Lemma 14]{HR}.

In the body of the paper, there appears the notation 
\begin{equation}\label{not rho}
\langle \mu,  \rho\rangle ,
\end{equation}
where $\mu$ is an element of $X_*(T)_{\Gamma_0, \BQ}=X_*(T)_{\Gamma_0}\otimes\BQ$.  To explain this notation, we recall that  $\tilde W$ contains canonically the affine Weyl group of a reduced root system $\Sigma$ such that $X_*(T)_{\Gamma_0}\otimes\BQ$ coincides with the vector space of the root system $\Sigma$.  Then, using the choice of a positive chamber made above, the first factor $\mu$ in \eqref{not rho} denotes a dominant representative and   $\rho$ denotes the half-sum of positive roots of $\Sigma$, cf. \cite[Prop. 4.21]{PRS}.
  \smallskip
  
  The $\sigma$-conjugacy classes of $G(\breve\BQ_p)$ are classified by Kottwitz in \cite{Ko1} and \cite{Ko2}. Denote their set by $B(G)$, i.e., $B(G)=G(\breve\BQ_p)/G(\breve\BQ_p)_\sigma$. We denote by $\nu$ the Newton map,
 \begin{equation}
 \nu\colon B(G)\to \big((X_*(T)_{\Gamma_0, \BQ})^+\big)^{\langle\sigma\rangle} ,
 \end{equation}
 comp. \cite[1.1]{HN2}. Here $(X_*(T)_{\Gamma_0, \BQ})^+$ denotes the intersection of  $X_*(T)_{\Gamma_0}\otimes \BQ$ with the set $X_*(T)_\BQ^+$ of dominant elements in $X_*(T)_\BQ$; the action of $\sigma$ on $(X_*(T)_{\Gamma_0, \BQ})/W_0$ is transferred to an action on $(X_*(T)_\BQ)^+$ ({\it L-action}). We denote by $\kappa$ the Kottwitz map,
 \begin{equation}
 \kappa\colon B(G)\to \pi_1(G)_\Gamma ,
 \end{equation}
 comp. \cite[(2.1)]{RV}. Here $\Gamma=\Gal(\ov\BQ_p/\BQ_p)$. 
 
 The set $B(G)$ is equipped with a partial order. For this, we note that there is a partial order on the set of dominant elements in $X_*(T)_\BQ$ (namely, the {\it dominance order}, i.e.,  $\nu\leq \nu'$ if $\nu'-\nu$ is a non-negative $\BQ$-sum of positive relative coroots). 
 We now define
 \begin{equation}\label{partordB}
 [b]\leq [b']\quad  \text{\it  if and only if }\quad  \kappa([b])=\kappa([b'])\,\,  \text{\it and }\,\, \nu([b])\leq \nu([b']) . 
 \end{equation}
 Here $\nu([b])$ and $ \nu([b'])$ denote the dominant representatives. 
 
 Let $\{\mu\}$ be a conjugacy class of cocharacters of $G$. Recall the finite subset $B(G, \{\mu\})$ of $B(G)$, consisting of {\it neutral acceptable} elements with respect to  $\{\mu\}$ in $B(G)$, cf. \cite{RV}. It is defined by 
 \begin{equation}
 B(G, \{\mu\})=\{ [b]\in B(G)\mid \kappa([b])=\mu^\natural, \nu([b])\leq \ov\mu \} .
 \end{equation}
 Here $\mu^\natural$ denotes the common image of $\mu\in\{\mu\}$ in $\pi_1(G)_\Gamma$, and $\ov\mu$ denotes the Galois average of a dominant representative of the image of an element  of $\{\mu\}$ in $X_*(T)_{\Gamma_0, \BQ}$  with respect to the L-action of $\sigma$ on $(X_*(T)_{\Gamma_0,\BQ})^+$.   The set $B(G, \{\mu\})$ inherits a partial order from $B(G)$. It has a unique minimal element, namely the unique basic element with image under $\kappa$ equal to $\mu^\natural$.  It also has a unique maximal element, determined in \cite{HN2}. 

\section{Axioms on integral models}\label{s:axioms}

\subsection{The set-up} Let $({\bf G}, \{h\})$ be a Shimura datum and let ${\bf K}=K^p K$  be an open compact subgroup of ${\mathbf G}(\BA_f)$, where $K^p \subset {\mathbf G}(\BA^p_f)$ and where $K=K_p$ is a parahoric subgroup of ${\mathbf G}(\BQ_p)$. Let $G={\mathbf G} \otimes_\BQ \BQ_p$ and let $\{\mu\}$ be the conjugacy class of cocharacters of $G$ corresponding to $\{h\}$. Here we use the opposite convention from Deligne \cite{De}: his $\mu$ is the inverse of ours. 
 
 Let ${\rm Sh}_{\mathbf K}={\rm Sh}({\mathbf G}, \{h\})_{\mathbf K}$ be the corresponding Shimura variety. It is a quasi-projective variety defined over the Shimura field $\bE$. 
We will postulate the existence of an integral model ${\mathbf S_{\bK}}$ over the ring of integers $O_E$ of the completion $E$ of $\bE$ at a place ${\bf p}$ above the fixed prime number $p$, with certain properties, which we  list below. Our aim is to  study the special fiber ${Sh}_K={\mathbf S}_{\bK}\times_{\Spec O_{\bE}}\Spec  \kappa_E$, resp. its set of geometric points,  and some stratifications on it.

\subsection{Basic axioms on integral models}

We now list our first set of axioms.

\begin{altenumerate}
\item Our first axiom concerns the change in the parahoric subgroup. 
 \begin{axiom}[Compatibility with changes in the parahoric] \label{ax funct}
  For any inclusion of parahoric subgroups $K\subset K'$, and setting $\bK=K^pK$ and $\bK'=K^pK'$, there is a natural morphism
\begin{equation}
\pi_{K, K'}\colon {\mathbf S}_{\bK}\to{\mathbf S}_{\bK'} ,
\end{equation}
which is proper and surjective, and is finite in the generic fibers. 
\end{axiom}

\item We postulate the existence of   a {\it  local model} $\bM^{{\rm loc}}_{K}$ attached to the triple $(G, \{\mu\}, K)$. Let $\CG=\CG_K$ be the group scheme over $\BZ_p$  corresponding to $K$. Then $\bM^{{\rm loc}}_{K}$ is a scheme which is projective and flat over $\Spec O_E$, equipped with an action of $\CG\otimes_{\BZ_p} O_E$, and with generic fiber equal to the partial flag variety associated to $(G, \{\mu\})$. Its formation should be  functorial in the parahoric subgroup $K$, i.e., for $K\subset K'$, there should be a proper and surjective morphism,
\begin{equation}\label{map pKK}
p_{K, K'}\colon \bM^{{\rm loc}}_{K}\to \bM^{{\rm loc}}_{K'} .
\end{equation}

Let $M^{{\rm loc}}_K$ be its special fiber.  Then $M^{{\rm loc}}_K $ is a projective variety over $\kappa_E$, with an action of $\CG_K\otimes_{\BZ_p}\kappa_E$.

\begin{axiom}[Existence of local models]\label{ax locmod}  
 {\it There is a smooth morphism of algebraic stacks  \cite[(7.1)]{R:guide}}
$$
{\boldsymbol \lambda}_K: {\mathbf S}_{\mathbf K} \to [\bM^{{\rm loc}}_K/\mathcal G_{O_E}] ,
$$
compatible with changes in the parahoric subgroup $K$. The action of $\CG_K\otimes_{\BZ_p}\kappa_E$ on $M^{{\rm loc}}_K $ has  finitely many orbits $\CO_w$ which are indexed by $w\in \Adm(\{\mu\})_K$. Furthermore, 
$$\CO_w\subset\ov{\CO}_{w'}\quad\text{ if and only if}\quad  w\leq w'$$
 in the partially ordered set $W_K\backslash\tilde W/W_K$.

{\rm Here $\CG=\CG_K$, and $\CG_{O_E}$ denotes its base change to $\Spec O_E$.}
\end{axiom}

 \begin{remark}
 Pappas and Zhu \cite{P-Z} have constructed such  local models under a tameness assumption on $G$. However, in their set-up, the orbits in $\bM^{{\rm loc}}_K$ are implicitly enumerated by a subset of the Iwahori Weyl group of a {\it loop group} version of $G(\breve \BQ_p)$. Axiom \ref{ax locmod} implicitly refers to Scholze's idea \cite{BS} that would construct local models of Shimura varieties whose special fibers are embedded as  closed subschemes of a \emph{Witt vector affine flag variety}. 

\end{remark}

We use the notation $\lambda_K$  for  the induced morphism of stacks on the special fiber $\lambda_K: Sh_K \to [ M^{{\rm loc}}_K /\CG_{\kappa_E}]$, but also for the map
\begin{equation}\label{def lambda}
\lambda_K: Sh_K \to W_K\backslash\tilde W/W_K ,
\end{equation} 
which associates to a point of $Sh_K$ the orbit of its image in $ [ M^{{\rm loc}}_K /\CG_{\kappa_E}]$.  
For any $w \in W_K\backslash\tilde W/W_K$, set 
\begin{equation}\label{ def KR}
KR_{K, w}=\lambda_K ^{-1}(\CO_w) \subset Sh_K ,
\end{equation} and call it the  {\it Kottwitz-Rapoport stratum} (KR stratum) of $Sh_K$ attached to $w$, cf. \cite[\S 8]{H}. It is a locally closed subvariety of $Sh_K$. Note that, by definition, $KR_{K, w}$ is non-empty  only if $w\in \Adm(\{\mu\})_K$.

\begin{remark}\label{dim KR}
Concerning the dimension of these strata, one may conjecture the following. Let $_K\tilde W^K$ be the set of elements of maximal length among all elements of minimal length in their right coset modulo $W_K$, cf. \cite[Prop. 4.20]{PRS}. Then  $_K\tilde W^K$ maps bijectively to $W_K\backslash\tilde W/W_K$. For $w\in W_K\backslash\tilde W/W_K$, let  $_K w^K\in {_K\tilde W^K}$ be the preimage of $w$ under this bijection.  The  KR stratum $KR_{K, w}$, if non-empty, should be smooth of dimension $\ell( _K w^K)$. 
\end{remark}

\item 
Recall $B(G)=G(\breve \BQ_p)/G(\breve \BQ_p)_\sigma$, the set of $\sigma$-conjugacy classes of $G(\breve \BQ_p)$. 

\begin{axiom} [Existence of a Newton stratification] \label{ax newton}
There is a map  
$$\delta_K: Sh_K\to B(G),
$$ 
compatible with changing the parahoric subgroup $K$ (i.e., with $\pi_{K, K'}$), and  such that for each $[b]\in B(G)$, the fiber   of $\delta_K$ over $[b]$ is the set of $\bar\kappa_E$-rational points of a locally closed subvariety $\mathit S_{K,[b]}$ of  $Sh_K$. Furthermore, if
$$
\mathit S_{K, [b]}\cap\ov{\mathit S}_{K, [b']}\neq\emptyset ,
$$
then $[b]\leq [b']$ in the sense of the partial order on $B(G)$, cf. \eqref{partordB}. 
\end{axiom}

The subvariety $\mathit S_{K, [b]}$ of $Sh_K$ is called  the  {\it Newton stratum }  of $Sh_K$ attached to $[b]$.

\begin{remarks} In the case of a hyperspecial parahoric subgroup $K$, the Newton strata $\mathit S_{K, [b]}$ should have the strong stratification property (the closure of a stratum is a union of strata), and for $[b], [b']\in B(G)$, one should have 
\begin{equation}\label{stratprop KR}
\mathit S_{K, [b]}\subset \ov{\mathit S}_{K, [b']}\quad \text{\it if and only if } \quad [b]\leq [b'] . 
\end{equation}
 Furthermore, the Newton strata $\mathit S_{K, [b]}$ should be equi-dimensional of dimension
\begin{equation}\label{defform}
\dim \mathit S_{K, [b]} = \langle  \mu+\nu([b]), \rho\rangle-\frac{1}{2}{\rm def}([b]) .
\end{equation}
 Here the first summand on the RHS is explained in \eqref{not rho}; for the second summand, comp., e.g., \cite{Ham}. These statements have been proved by Hamacher in the PEL case \cite{Ham}. 

These properties do not extend to general parahoric subgroups. For instance, the strong stratification property fails in the unramified quadratic Hilbert-Blumenthal case for the Iwahori subgroup, cf. \cite{St}. The equi-dimensionality fails for the basic stratum for the Iwahori subgroup in the Siegel case when $g=2$ \cite[Prop. 6.3]{Yu}. 

\end{remarks}

\end{altenumerate}
\subsection{Joint stratification and basic non-emptyness}

Let $\breve K_\sigma \subset \breve K \times \breve K$ be the graph of the Frobenius map $\sigma$ and $G(\breve\BQ_p)/\breve K_\sigma$ be the set of $\breve K$-$\sigma$-conjugacy classes on $G(\breve\BQ_p)$. The embedding $\breve K_\sigma \subset G(\breve\BQ_p)_\sigma$ induces a projection map 
\begin{equation}\label{pr_B}
d_K\colon G(\breve\BQ_p)/\breve K_\sigma \to B(G) .
\end{equation} 
On the other hand, the embedding $\breve K_\sigma \subset \breve K \times \breve K$ induces a map 
\begin{equation}\label{pr_W}
\ell_K\colon G(\breve\BQ_p)/\breve K_\sigma \to \breve K \backslash G(\breve\BQ_p)/\breve K .
\end{equation}
We now add the following axioms to our list.

\smallskip

\begin{altenumerate}
\item The first axiom relates the two maps $ \lambda$ and $\delta$ introduced in Axioms \ref{ax locmod} and \ref{ax newton}. Note that in its formulation, we identify $ \breve K \backslash G(\breve\BQ_p)/\breve K$ with $ W_K \backslash\tilde W/W_K$, cf. \cite[Prop. 8]{HR}. 
\begin{axiom}[Joint stratification]\label{ax Y}
a) There exists a natural map 
$$
\Upsilon_K: Sh_K  \to G(\breve\BQ_p)/\breve K_\sigma ,
$$
 compatible with changes in the parahoric subgroup $K$,  such that the following diagram commutes \[\xymatrix{ & & \breve K \backslash G(\breve\BQ_p)/\breve K \\  Sh_K  \ar[r]^-{\Upsilon_K} \ar@/^1pc/[urr]^{ \lambda_K } \ar@/_1pc/[drr]_{ \delta_K } & G(\breve\BQ_p)/\breve K_\sigma \ar[ur]_{\ell_K} \ar[dr]^{d_K} & \\ & & B(G)}.\]
 {\rm Here the map $\lambda_K$ is the map \eqref{def lambda}.}
 \smallskip
 
\noindent b)  Furthermore,
 $$
 \Im\Upsilon_K=\ell_K^{-1}(\Im  \lambda_K ) .
 $$
 
 \smallskip
 
 \noindent c) For $K\subset K'$, and any element $y\in \Im(\Upsilon_{K})$ with image $y'\in G(\breve\BQ_p)/\breve K'_\sigma$, the natural map
 $$
 {\pi_{K, K'}}_{| \Upsilon_{K}^{-1}(y)}\colon \Upsilon_{K}^{-1}(y)\to \Upsilon_{K'}^{-1}(y')
 $$
 is surjective with finite fibers. 
\end{axiom}
It should be pointed out that parts b) and c) of this axiom are principally  used in connection with the study of EKOR strata in section \ref{s:EKOR}; more precisely, if b) and c) are omitted, the only change outside section \ref{s:EKOR} is that the equality sign  in Corollary \ref{imUps} has to be replaced by an inclusion sign $\subseteq$. 
\begin{remarks}\label{rem ax2}
(1) The images of $\lambda_K$ and of $\delta_K$ are finite, comp. Proposition \ref{incl} below; by  Axiom \ref{ax Y} b), the image of $\Upsilon_K$ is infinite. 

(2) Axiom Axiom \ref{ax Y} c) for $K\subset K'$ follows from  Axiom \ref{ax Y} c) for $I\subset K'$. This follows  from the surjectivity property of $\pi_{I, K'}$ in Axiom \ref{ax funct}. Also, it is clear that  Axiom \ref{ax Y} c)  for $K\subset K'$ and for $K'\subset K''$ implies Axiom \ref{ax Y} c)  for $K\subset K''$.

(3) The fibers of $\Upsilon_K$ are the group-theoretic version of Oort's {\it central leaves} \cite{Oo1}. It seems reasonable to expect the fibers of $\Upsilon_K$ to be closed subsets of the corresponding Newton stratum (this is what Oort proves in the Siegel case when $K$ is hyperspecial). Furthermore, the fibers should be smooth  and equi-dimensional  with dimension given 
$$
\dim \Upsilon_K^{-1}(y)=\langle \nu({d_K(y)}), 2\rho\rangle , 
$$
comp.  \cite{Ham}. The RHS is defined in \eqref{not rho}.  

It may also be conjectured that the morphism in c) above is a finite morphism which is the composition of a radicial morphism and a  finite \'etale morphism. 

 \end{remarks}
\item The second axiom is a weak non-emptiness statement.  Recall from \eqref{def tau} the element $\tau=\tau_{\{\mu\}}$ of length zero in $\tilde W$. 
\begin{axiom}[Basic non-emptiness]\label{ax nonemptytau}
 The map 
$$
KR_{I, \tau} \to \pi_0(Sh_I)
$$ is surjective.  
\end{axiom}
\end{altenumerate}
Here $\pi_0(Sh_K)$ denotes the set of geometric connected components of $Sh_K$. In other words, this axiom postulates that every geometric connected component of $Sh_I$ intersects the KR stratum $KR_{I, \tau}$. 
\begin{remark}\label{rem comp}
In particular, Axiom \ref{ax nonemptytau} states that $KR_{I, \tau}$ is non-empty. The converse can sometimes be proved if a good theory of compactifications exists. Indeed, we would then have a  $G(\BA_f)$-equivariant identification  $\pi_0(Sh_I)=\pi_0({\rm Sh}_{\mathbf K})$, where $\mathbf K=K^p I$. Hence, if $G(\BA_f^p)$ acts transitively on $\varprojlim\nolimits_{K^p}\pi_0({\rm Sh}_{\mathbf K})$,  the non-emptiness of $KR_{I, \tau}$ would imply  Axiom \ref{ax nonemptytau}. 
\end{remark}

The following lemma gives a relation between Axiom \ref{ax Y} b) for a parahoric and a larger parahoric. 
\begin{lemma}\label{prop KK'} We assume Axioms \ref{ax funct} and \ref{ax locmod}. 
Let $K\subset K'$. If $\Im(\Upsilon_K)=\ell_K^{-1}(\Adm(\{\mu\})_K)$, then $\Im(\Upsilon_{K'})=\ell_{K'}^{-1}(\Adm(\{\mu\})_{K'})$.
\end{lemma}
\begin{proof}  
It suffices to see that the natural map 
$$
\ell_K^{-1}(\Adm(\{\mu\})_K)=\big(\breve K \Adm(\{\mu\}) \breve K\big)/\breve K_\sigma\to \big(\breve K' \Adm(\{\mu\}) \breve K'\big)/\breve K'_\sigma=\ell_{K'}^{-1}(\Adm(\{\mu\})_{K'}) .
$$ is surjective. This follows from \cite[\S 6.3 (b)]{H2}.
\end{proof}

We note some first consequences of these axioms. As a preliminary, we mention the following result (conjectured in \cite{KoR} and \cite{R:guide}). 

\begin{theorem}[{cf. \cite[Theorem A]{H2}}]\label{K-R}
Let $K$ be a parahoric subgroup and $[b] \in B(G)$. Then 
\begin{flalign*}\phantom{\qed} & & [b] \cap \big( \cup_{w \in \Adm(\{\mu\})_K} \breve K w \breve K\big)\neq \emptyset\quad  \text{ if and only if }\quad [b] \in B(G, \{\mu\}) .& & \qed\end{flalign*}
\end{theorem}

Using this theorem, we obtain the following result.

\begin{proposition}\label{incl}
 There are the following inclusions,  
\begin{altenumerate}
\item $\Im(\lambda_K) \subset \Adm(\{\mu\})_K ,$
\item  $\Im(\delta_K) \subset B(G, \{\mu\})$.
\end{altenumerate}
\end{proposition}
\begin{proof}  

Here (i) is just a restatement of the remark right after Axiom \ref{ax locmod}. To see (ii), we note that, by Axiom \ref{ax Y} a),  the image of $\Upsilon_K$ is contained in $\cup_{w \in \Adm(\{\mu\})_K} \breve K w \breve K/\breve K_\sigma$ and hence 
$$\Im(\delta_K)\subset  \big\{[b] \in B(G)\mid [b] \cap (\cup_{w \in \Adm(\{\mu\})_K} \breve K w \breve K)\neq \emptyset\big\}.$$ 
Hence the assertion follows from the ``only if'' direction of Theorem \ref{K-R}. 
\end{proof}

\section{Non-emptiness of  KR strata}\label{s:KR}

In this section we prove the nonemptiness of  KR strata. 

\begin{theorem}\label{nonempty-KR} Let $K$ be a parahoric subgroup and let $X_K$ be a geometric connected component of $Sh_K$. Then
$$\lambda_K(  X_K  )=\Adm(\{\mu\})_K .$$
{\rm  In other words, any geometric connected component of $Sh_K$ intersects any KR stratum  (as their indices run over  their natural range, i.e.,  $\Adm(\{\mu\})_K$). }
\end{theorem}

\begin{proof}
We first consider the case where $K=I$, the Iwahori subgroup. By Axiom \ref{ax locmod}, $\lambda_I$ is smooth, and hence is open. By Axiom \ref{ax nonemptytau}, $KR_{I, \tau} \cap X_I \neq \emptyset$. Since $\breve I \tau \breve I/\breve I$ is the unique closed $\breve I$-orbit in $\cup_{w \in \Adm(\{\mu\})} \breve I w \breve I/\breve I$, we conclude that $\lambda_I(  X_I )=  M_I^{{\rm loc}} /\breve I$. Hence the assertion  holds for $K=I$.  

Now we consider the  case of a general parahoric. We use the commutative diagram

\[\xymatrix{
 \pi_{I, K}^{-1}(X_K) \ar[r]^-{ \lambda_{I}}\ar[d]^-{\pi_{I, K}} & \Adm(\{\mu\})\ar[d] \\ 
  X_K  \ar[r]^-{ \lambda_K} &  \Adm(\{\mu\})_K.}
\] 
We just proved that $\lambda_I  \vert \pi_{I, K}^{-1}(X_K)$ is surjective. The map $\Adm(\{\mu\})\to\Adm(\{\mu\})_K$ is surjective by definition. Hence the assertion for $K$ follows from the commutativity of the diagram. 
\end{proof}
\begin{corollary}\label{imUps}
The image of $\Upsilon_K$ is given by
$$
\Im \Upsilon_K= \cup_{w \in \Adm(\{\mu\})_{K}} \breve K w \breve K/\breve K_{\sigma} .
$$
\end{corollary}
\begin{proof}
This follows from Theorem \ref{nonempty-KR} and Axiom \ref{ax Y} b). \end{proof}
\begin{remark}
In the proof of this corollary, only the weakening $KR_{I, \tau}\neq\emptyset$, i.e., $\tau\in\Im(\lambda_I)$, of Axiom \ref{ax nonemptytau} is used. 

\end{remark}
\section{Newton strata}\label{s: N}

In this section, we study  Newton strata and use the axioms to prove the nonemptiness of Newton strata and  their closure relations. Our approach is based on the relation between certain conjugacy classes in the Iwahori-Weyl group $\tilde W$ and the $\sigma$-conjugacy classes of $G(\breve \BQ_p)$. 

\subsection{$\sigma$-straight elements} 
Note that $\tilde W$ is equipped with a natural action induced from $\sigma$. We regard $\sigma$ as an element in the group $\tilde W \rtimes \langle \sigma \rangle$. The length function on $\tilde W$ extends in a natural way to a length function on $\tilde W \rtimes \langle \sigma \rangle$ by requiring $\ell(\sigma)=0$. 

For any $w \in \tilde W$, we choose a representative in $N(\breve \BQ_p)$ and still denote it by $w$. The restriction of the Newton map $\nu$ to $\tilde W$ can be described explicitly as follows.

Recall that $\tilde W=X_*(T)_{\Gamma_0} \rtimes W_0$.  For any $w \in \tilde W$, there exists $n \ge 1$ such that $\sigma^n$ acts trivially on $\tilde W$ and that $\lambda=(w \sigma)^n=w \sigma(w) \cdots \sigma^{n-1}(w) \in X_*(T)_{\Gamma_0}$. The element $\frac{1}{n}\lambda \in (X_*(T)_\BQ)^{\Gamma_0}$ is independent of the choice of $n$. Then  $\nu(w)$ is the unique dominant element in the $W_0$-orbit of $\frac{1}{n}\lambda$. Note that $\nu(w)$ is independent of the choice of representative of $w$ in $N(\breve \BQ_p)$ and is constant on $\sigma$-conjugacy classes in $\tilde W$. 

\smallskip

Now we recall the definition of $\sigma$-straight elements and $\sigma$-straight conjugacy classes \cite{HN}. 
An element $w \in \tilde W$ is called {\it $\sigma$-straight} if $\ell(w \sigma(w) \sigma^2(w) \cdots \sigma^{m-1}(w))=m \ell(w)$ for all $m \in \BN $, i.e., $\ell((w \sigma)^m)=m \ell(w)$ for all $m \in \BN$. By \cite[2.4]{H1}, $w$ is $\sigma$-straight if and only if $\ell(w)=\langle\nu(w), 2 \rho\rangle$. For the last notation, comp. \eqref{not rho}. We call a $\sigma$-conjugacy class of $\tilde W$ $\sigma$-straight if it contains a $\sigma$-straight element,  and denote by  $B(\tilde W)_{\sstr}$  the set of $\sigma$-straight $\sigma$-conjugacy classes of $\tilde W$.

\smallskip

We have the following results on the relation between $\sigma$-straight elements in $\tilde W$ and $B(G)$. 

\begin{theorem}{\rm  a) ({cf. \cite[Theorem 3.7]{H1}})}\label{str-bgmu}
For any $\sigma$-straight element $w$, $\breve I w \breve I$ is contained in a single $\sigma$-conjugacy class of $G(\breve \BQ_p)$.

\smallskip

{\rm b) ({cf. \cite[Theorem 3.3]{H1}})}
The map $$\Psi: B(\tilde W)_{\sstr} \to B(G)$$ induced by the inclusion $N(T)(\breve \BQ_p) \subset G(\breve \BQ_p)$ is bijective.

\smallskip

 {\rm c) ({cf. \cite[Proposition 4.1]{H2}})}
Let $\Adm(\{\mu\})_{\sstr}$ be the set of $\sigma$-straight elements in the admissible set $\Adm(\{\mu\})$. Then   $\Psi$ maps  the image of $\Adm(\{\mu\})_{\sstr}$ in $B(\tilde W)_{\sstr}$ bijectively  to $ B(G, \{\mu\}) .$ \qed
\end{theorem}
\begin{remark}
Note that the point c) is closely related to Theorem \ref{K-R}. 
\end{remark}

\subsection{Non-emptiness}
As an application of Theorem \ref{str-bgmu}, a), we obtain the following fact. 
\begin{proposition}\label{KR-N}
Let $w$ be a $\sigma$-straight element, then $KR_{I, w} \subset \mathit S_{I, [w]}$.
 \qed
\end{proposition}
\noindent Here $[w]$ denotes the $\sigma$-conjugacy class of $w$, resp. its image under $\Psi$. 
\smallskip

Now we prove the non-emptiness of Newton strata. 
\begin{theorem}\label{nonempty-N} Let $K$ be a parahoric subgroup and let $X_K$ be a geometric connected component of $Sh_K$. Then
  $$\delta_K(  X_K  )=B(G, \{\mu\}).$$
{ \rm In other words, any geometric connected component of $Sh_K$ intersects any Newton stratum (as their indices run over  their natural range, i.e.,   $B(G, \{\mu\})$. }
\end{theorem}
\begin{proof} The inclusion $\delta_K(  X_K  )\subset B(G, \{\mu\})$ is the content of Proposition \ref{incl}, (ii). 

Now let $[b] \in B(G, \{\mu\})$. By Theorem \ref{str-bgmu} c), there exists a $\sigma$-straight element $w \in \Adm(\{\mu\})$ such that $w \in [b]$. By Proposition \ref{KR-N}, $KR_{I, w} \subset \mathit S_{I, [b]}$. By Theorem \ref{nonempty-KR}, $KR_{I, w} \cap X_I \neq \emptyset$. Hence $\mathit S_{I, [b]} \cap X_I \neq \emptyset$. Therefore the  assertion holds for $K=I$. 

The case of a general parahoric $K$ follows  from the commutative diagram
 
 \begin{equation*}
 \begin{xy}
 \xymatrix@C-1em{
  \pi_{I, K}^{-1}(X_K) \ar[dd]^{\pi_{I, K}}\ar[rd]^{\delta_I} &\\         & B(G) . \\
				    X_K   \ar[ur]_{\delta_K}  &           		                            
}
\end{xy}
\end{equation*}

\end{proof}

\subsection{Closure relation} 

We recall the partial order on $B(\tilde W)_{\sstr}$ introduced in \cite[\S 3.2]{H2}. For $\mathcal O, \mathcal O' \in B(\tilde W)_{\sstr}$, we say that $\mathcal O' \preceq \mathcal O$ if for some (or equivalently\footnote{This equivalence follows from a remarkable property of the $\sigma$-straight conjugacy classes (see \cite[Theorem 3.8]{HN}); the transitivity of the partial order is deduced from this equivalence.}, any) $\sigma$-straight element $w \in \mathcal O$, there exists a $\sigma$-straight element $w' \in \mathcal O'$ such that $w' \le w$ (the Bruhat order on $\tilde W)$. 

The natural bijection in Theorem \ref{str-bgmu} b) is in fact a bijection of posets in the following sense. 

\begin{theorem}[{cf. \cite[Theorem B]{H2}}]\label{posets}
Let $\mathcal O,  \mathcal O' \in B(\tilde W)_{\sstr}$. Then $\mathcal O' \preceq \mathcal O$ if and only if $\Psi(\mathcal O') \le \Psi(\mathcal O)$. \qed
\end{theorem}

\smallskip

Now we prove the following closure relation between Newton strata (this kind of statement is sometimes referred to as {\it  Grothendieck's conjecture}).

\begin{theorem}\label{grothconj}
Let $K$ be a parahoric subgroup. Let $[b], [b'] \in B(G, \{\mu\})$. Then  $\overline{\mathit S}_{K,[b']} \cap \mathit S_{K, [b]} \neq \emptyset$ if and only if $[b] \le [b']$. 
\end{theorem}

\begin{proof}
The ``only if'' direction is part of Axiom \ref{ax newton}. Now we prove the ``if'' direction.  Using the properness of $\pi_{I, K}$, it suffices to consider the case $K=I$. 

By Theorem \ref{str-bgmu} c), there exists a $\sigma$-straight element $w \in \Adm(\{\mu\})$ such that $w \in [b]$. By Corollary \ref{KR-N}, $KR_{I, w} \subset \mathit S_{I, [b]}$. Then $\overline{\mathit S}_{I, [b]}\supset \overline{KR}_{I, w}=\sqcup_{w' \le w} KR_{I, w'}$. By Theorem \ref{posets}, there exists a $\sigma$-straight element $w'$ such that $w' \in [b']$ and $w' \le w$. This finishes the proof. 
\end{proof}
\begin{remark} We used that if $\overline{\mathit S}_{K,[b']} \cap \mathit S_{K, [b]} \neq \emptyset$, then also $\overline{\mathit S}_{K',[b']} \cap \mathit S_{K', [b]} \neq \emptyset$  for any $K'\supset K$. The converse also holds, since $\pi_{K, K'}$ is proper and surjective. 
\end{remark}

\section{EKOR strata}\label{s:EKOR}

\subsection{Definition of $ \upsilon_K $}\label{ss:upsilon} Let $K$ be a parahoric subgroup, and let $\breve K_1$ be the pro-unipotent radical of $\breve K$. Then $$\breve K_\sigma \subset \breve K_\sigma (\breve K_1 \times \breve K_1) \subset \breve K \times \breve K.$$ Thus  $\lambda_K$  factors through the composition of the following two maps,  
$$
 Sh_K  \to G(\breve\BQ_p)/\breve K_\sigma \to G(\breve\BQ_p)/\breve K_\sigma (\breve K_1 \times \breve K_1),
$$
 where the first map is $\Upsilon_K$ and the second map is the natural projection map. We denote  the composition map by 
\begin{equation}
 \upsilon_K :  Sh_K  \to G(\breve\BQ_p)/\breve K_\sigma (\breve K_1 \times \breve K_1) .
\end{equation} 
We therefore obtain a commutative diagram
 
 \begin{equation*}
 \begin{xy}
 \xymatrix@C-1em{
  Sh_K  \ar[dd]_{ \upsilon_K }\ar[rd]^{ \lambda_K } &\\         &    \breve K\backslash G(\breve\BQ_p)/\breve K  . \\
			G(\breve\BQ_p)/\breve K_\sigma (\breve K_1 \times \breve K_1) 	  \ar[ur]_{{\rm nat}_K}  &           		                            
}
\end{xy}
\end{equation*}

\subsection{$G$-stable piece decomposition} Now we discuss the decomposition of $\breve K w \breve K$ into finitely many subsets stable under the action of $\breve K_\sigma$,  analogous to the $G$-stable piece decomposition (for reductive groups $G$ over algebraically closed fields) introduced by Lusztig in \cite{par-I}. 

\begin{theorem}\label{kwk}
Let $K$ be a parahoric subgroup. Then
\smallskip

(a) For any $x \in {}^K \tilde W$, $\breve K_\sigma(\breve K_1 x \breve K_1)=\breve K_\sigma (\breve I x \breve I)$.

\smallskip

(b) $G(\breve \BQ_p)=\sqcup_{x \in {}^K \tilde W} \breve K_\sigma(\breve K_1 x \breve K_1)=\sqcup_{x \in {}^K \tilde W} \breve K_\sigma (\breve I x \breve I)$.
\end{theorem}
This result is essentially contained in \cite[1.4]{L3} and \cite[Proposition 2.5 \& 2.6]{He-11}. We include a proof for completeness.

\begin{proof} 

Let $\tilde{\mathbb S}$ be the set of simple reflections in  $\tilde W$ and $J \subset \tilde{\mathbb S}$ be the set of simple reflections in $W_K$. Since our parahoric subgroup $\breve K$ comes from $K$ over $\BQ_p$, we have  $\sigma(J)=J$. 

Let $w \in {}^J \tilde W^J$, i.e., $w$ is of shortest length in $W_K w W_K$. It suffices to show that 
\begin{equation}\label{6-2}
\breve K w \breve K=\sqcup_{x \in W_K w W_K \cap {}^K \tilde W} \breve K_\sigma (\breve K_1 x \breve K_1)
\end{equation} and 
\begin{equation}\label{6-3}
 \breve K_\sigma (\breve K_1 x \breve K_1)=\breve K_\sigma (\breve I x \breve I) ,\quad \forall x \in W_K w W_K \cap {}^K \tilde W.
\end{equation}

Set $\sigma'=\sigma \circ \text{Ad}(w)$. Then we obtain  $\sigma': (\breve K \cap w^{-1} \breve K w) \to (\breve K \cap w \breve K w^{-1})$. The map $\breve K \to \breve K w \breve K, k \mapsto w k$ induces a bijection 
\begin{equation}\label{identshiftw}
\breve K/(\breve K \cap w^{-1} \breve K w)_{\sigma'} \to \breve K w \breve K/\breve K_\sigma.
\end{equation}

Let $\overline{\breve K}=\breve K/\breve K_1$ be the reductive quotient of $\breve K$. Let $\overline{B}$ be the image of $\breve I$ in $\overline{\breve K}$ and $\overline{T}$ be the maximal torus of the Borel subgroup $\overline{B}$. Let $J_1=J \cap \Ad(w)^{-1}(J)$. 

By \cite[Theorem 2.8.7]{Car}, the image of $\breve K \cap w^{-1} \breve K w$ in $\overline{\breve K}$ is of the form $\overline{L}_{J_1} U$, where $\overline{L}_{J_1}$ is the standard Levi subgroup of type $J_1$ of $\overline{\breve K}$ and $U$ is a connected subgroup in the unipotent radical $U_{\overline{P}_{J_1}}$ of the standard parabolic subgroup $\overline{P}_{J_1}$  of type $J_1$ of $\overline{\breve K}$. We obtain a natural map 
\begin{equation}\label{reductmap}
f\colon  \breve K w \breve K/\breve K_\sigma \to \overline{\breve K}/({\overline{L}_{J_1}})_{\sigma'} (U_{\overline{P}_{J_1}} \times U_{\overline{P}_{\sigma'(J_1)}}) .
\end{equation}
By \cite[1.6 (c)]{L3}, the map $f$ in \eqref{reductmap} factors through a bijection 
\begin{equation*}
\breve K w \breve K/\breve K_\sigma (\breve K_1 \times \breve K_1)\to  \overline{\breve K}/({\overline{L}_{J_1}})_{\sigma'} (U_{\overline{P}_{J_1}} \times U_{\overline{P}_{\sigma'(J_1)}}) .
\end{equation*} 
By \cite[3.1(b) \& (c)]{He-09},  the underlying space of $ \overline{\breve K}/({\overline{L}_{J_1}})_{\sigma'} (U_{\overline{P}_{J_1}} \times U_{\overline{P}_{\sigma'(J_1)}})$ is the finite set ${}^{J_1} W_K$. By \cite[2.1]{par-I}, $w ({}^{J_1} W_K)=W_K w W_K \cap {}^K \tilde W$. Hence (\ref{6-2}) is proved. 

By \cite[2.1 \& Corollary 2.6]{He-09}, for any $x \in {}^{J_1} W_K$, $$({\overline{L}_{J_1}})_{\sigma'} (U_{\overline{P}_{J_1}} x U_{\overline{P}_{\sigma'(J_1)}})=({\overline{L}_{J_1}})_{\sigma'} (\overline{B} x \overline{B}).$$ Its inverse image under $f$ is $\breve K_\sigma (\breve K_1 x \breve K_1)=\breve K_\sigma (\breve I x \breve I)$. Hence (\ref{6-3}) is proved. 
\end{proof}

\begin{corollary}\label{natI}
The map ${\rm nat}_I$ is bijective. \qed
\end{corollary}
Let us identify $\breve K\backslash G(\breve\BQ_p)/\breve K $ with $W_K\backslash\tilde W/W_K$. We have the following commutative diagram,

\[\xymatrix{
G(\breve\BQ_p)/\breve K_\sigma (\breve K_1 \times \breve K_1)\ar[d]^-{{\rm nat}_K}\ar[r] & {}^K \tilde W \ar[d]^-{{\rm nat}_K} \\ 
 \breve K\backslash G(\breve\BQ_p)/\breve K  \ar[r]&   W_K\backslash \tilde W/W_K .}
\]

Here the horizontal arrows are bijective and the map ${\rm nat}_K\colon {}^K \tilde W \to W_K\backslash \tilde W/W_K$ is the natural projection  sending an element to its double coset. 
Using Theorem \ref{kwk}, we may rewrite the diagram in the previous subsection as 

 \begin{equation*}
 \begin{xy}
 \xymatrix@C-1em{
  Sh_K  \ar[dd]_{ \upsilon_K }\ar[rd]^{ \lambda_K } &\\         &    W_K\backslash \tilde W/W_K  . \\
			^K\tilde W 	  \ar[ur]_{{\rm nat}_K}  &           		                            
}
\end{xy}
\end{equation*}

\begin{remark}
As a consequence of Proposition \ref{incl} and Theorem \ref{kwk}, the image of $\upsilon_K$ is finite. 
\end{remark}
\subsection{Definition of EKOR strata} By Axiom \ref{ax locmod} and Theorem \ref{kwk}, the image of $  \upsilon_K $ is contained in the set  
$\Adm(\{\mu\})^K \cap {}^K \tilde W.$ 
In fact, by Corollary \ref{imUps} (which uses Axiom \ref{ax Y} b)), this set is {\it equal} to the image of $\upsilon_K$. 
\begin{definition} The  {\it Ekedahl-Kottwitz-Oort-Rapoport stratum} (EKOR stratum) of $  Sh_K $ attached to $x \in {}^K \tilde W$ is the subset
 \[
   EKOR_{K, x} = \upsilon_K ^{-1}(x) \subset   Sh_K .
 \]  
\end{definition}
Hence $EKOR_{K, x}$ is non-empty only if $x\in \Adm(\{\mu\})^K $. 
We will prove in Theorem \ref{EKORlocclosed} that $EKOR_{K, x}$ is a locally closed subset. 
\begin{remarks}   (1) For a general parahoric subgroup,  the EKOR stratification is finer than the KR stratification (the map $ \lambda_K $ factors through $ \upsilon_K $).

(2) If $G$ is unramified and $K$ is hyperspecial, the definition of the EKOR stratification coincides with the Ekedahl-Oort stratification in the sense of Viehmann \cite{Vi}.
If $K=I$ is the Iwahori subgroup then, by Corollary \ref{natI}, the EKOR strata coincide with the KR strata.  Therefore the  EKOR stratification for  a  general parahoric subgroup  interpolates between the EO stratification for the hyperspecial case and the KR stratification for the Iwahori case.
 \end{remarks}

\subsection{Change of  parahoric}
Now we discuss the relation between the EKOR strata for different parahoric subgroups. To do this, we need the following result.

\begin{proposition}\label{new}
Let $K$ be a standard parahoric subgroup. For any $w \in \tilde W$, there exists a subset $\Sigma_K(w)$ of $W_K w W_K \cap {}^K \tilde W$ such that $$\breve K_\sigma (\breve I w \breve I)=\sqcup_{x \in \Sigma_K(w)} \breve K_\sigma (\breve I x \breve I).$$
Moreover, if $w \in {}^K \tilde W$, then $\Sigma_K(w)=\{w\}$. 

{\rm In general, $\Sigma_K(w)$ may contain more than one element. }
\end{proposition}

The proof uses the ``partial conjugation method'' of \cite{He-07}. We first introduce some notation. 

Let $J \subset \tilde{\mathbb S}$. For $w, w' \in \tilde W$ and $s \in J$, we write $w \xrightarrow{s}_{J, \sigma} w'$ if $w'=s w \sigma(s)$ and $\ell(w') \le \ell(w)$. We write $w \to_{J, \sigma} w'$ if there exists a finite sequence $w=w_0, w_1, \cdots, w_n=w'$ and $s_1, \cdots, s_n \in J$ such that $w_0 \xrightarrow{s_1}_{J, \sigma} w_1 \xrightarrow{s_2}_{J, \sigma} \cdots \xrightarrow{s_n}_{J, \sigma} w_n$.

For $w \in {}^J \tilde W$, we write $\text{Ad}(w) \sigma(J)=J$ if for any simple reflection $s \in J$, there exists a simple reflection $s' \in J$ such that $w \sigma(s) w^{-1}=s'$. In this case, $w \in {}^J \tilde W^{\sigma(J)}$. It is easy to see that for any $J_1, J_2 \subset \tilde{\mathbb S}$, and $w \in {}^{J_1 \cup J_2} \tilde W$, with $\text{Ad}(w) \sigma(J_1)=J_1$ and $\text{Ad}(w) \sigma(J_2)=J_2$, it follows that  $\text{Ad}(w) \sigma(J_1 \cup J_2)=J_1 \cup J_2$.  Thus for any $J \subset \tilde{\mathbb S}$ and $w \in {}^K \tilde W$, the set $\{J' \subset J\mid  \text{Ad}(w) \sigma(J')=J'\}$ contains a unique maximal element. We denote it by $I(J, w, \sigma)$. 

We will use the following result \cite[Proposition 3.4]{He-07} (see also \cite[Theorem 2.5]{HN}).
\begin{proposition}\label{partial-conj}
Let $w \in \tilde W$. For any $J \subset \tilde{\mathbb S}$, there exists $x \in {}^J \tilde W$, and an element $u$ in the Weyl group $W_{I(J, x, \sigma)}$ such that $w \to_{J, \sigma} u x$. \qed
\end{proposition}

We also have the following results. The proofs are similar to the proofs of \cite[Lemma 3.1 \& Lemma 3.2]{H1} and we omit them here.

\begin{lemma}\label{new1}
Let $K$ be a standard parahoric subgroup. Let $w \in \tilde W$ and $s \in J$. Then 

(1) If $\ell(s w \sigma(s))=\ell(w)$, then $\breve K_\sigma (\breve I w \breve I)=\breve K_\sigma (\breve I s w \sigma(s) \breve I)$. 

(2) If $\ell(s w \sigma(s))<\ell(w)$, then $\breve K_\sigma (\breve I w \breve I)=\breve K_\sigma (\breve I s w \sigma(s) \breve I) \cup K_\sigma (\breve I s w \breve I)$. \qed
\end{lemma}

\begin{lemma}\label{new2}
Let $K$ be a standard parahoric subgroup and $J$ be the set of simple reflections in $W_K$. Let $x \in {}^J \tilde W^{\sigma(J)}$ with $\text{Ad}(x) \sigma(J)=J$. Then for any $u \in W_K$, we have 
\begin{flalign*}\phantom{\qed} & &\breve K_\sigma (\breve I u x \breve I)=\breve K_\sigma (\breve I x \breve I).& & \qed\end{flalign*}
\end{lemma}

\begin{proof}[Proof of Proposition \ref{new}]
We argue by induction on $\ell(w)$. 
Let $J$ be the set of simple reflections in $W_K$. Let $x \in {}^J \tilde W$ and $u \in W_{I(J, x, \sigma)}$ with $w \to_{J, \sigma} (u x)$. 

If $\ell(w)=\ell(u x)$, then by Lemma \ref{new1} (1), $\breve K_\sigma (\breve I w \breve I)=\breve K_\sigma (\breve I u x \breve I)$. Let $K' \subset K$ be the standard parahoric subgroup corresponding to $I(J, x, \sigma)$. Then by Lemma \ref{new2}, $$\breve K_\sigma(\breve I u x \breve I)=\breve K_\sigma (\breve K'_{\sigma} (\breve I u x \breve I))=\breve K_\sigma (\breve K'_{\sigma} (\breve I x \breve I))=\breve K_\sigma (\breve I x \breve I).$$ 

If $\ell(w)>\ell(u x)$, then by the definition of $\to_{J, \sigma}$, there exists $w' \in \tilde W$ and $s \in J$ such that $w \to_{J, \sigma} w'$ and $\ell(w)=\ell(w')>\ell(s w' \sigma(s))$. By Lemma \ref{new1}, $$\breve K_\sigma (\breve I w \breve I)=\breve K_\sigma (\breve I w' \breve I)=\breve K_\sigma (\breve I s w' \sigma(s) \breve I) \cup \breve K_\sigma (\breve I s w' \breve I).$$ Now the statement follows from inductive hypothesis on $s w' \sigma(s)$ and on $s w'$. 
\end{proof}

In fact, by Proposition \ref{partial-conj}, the subset $\Sigma_K(w)$ can be determined inductively as follows:
\begin{altitemize}
\item If $x \in {}^J \tilde W$ and $u \in W_{I(J, x, \sigma)}$, then $\Sigma_K(u x)=\{x\}$. 
\item If  $w \in \tilde W$ and $s \in J$ with $\ell(s w \sigma(s))=\ell(w)$, then $\Sigma_K(w)=\Sigma_K(s w \sigma(s))$. 
\item If $w \in \tilde W$ and $s \in J$ with $\ell(s w \sigma(s))<\ell(w)$, then $\Sigma_K(w)=\Sigma_K(s w \sigma(s)) \cup \Sigma_K(s w)$.

\end{altitemize}

We also make use of the following result  \cite[Theorem 6.1]{H2} (see also \cite[Proposition 5.1]{HH} for a different proof). 

\begin{theorem}\label{comp}
For any standard parahoric subgroup $K$, 
\begin{flalign*}\phantom{\qed} & & \Adm(\{\mu\})^K \cap {}^K \tilde W=\Adm(\{\mu\}) \cap {}^K \tilde W.& & \qed\end{flalign*}
\end{theorem}

As a consequence, if $K' \subset K$, then the index set $\Adm(\{\mu\})^K \cap {}^K \tilde W$ for the EKOR strata with level $K$ is contained in the index set $\Adm(\{\mu\})^{K'} \cap {}^{K'} \tilde W$ for the EKOR strata with level $K'$ (the smaller the parahoric, the bigger the index set). In the sequel, we  identify  the index set for $K$ with  a subset of the index set for $K'$. 

Now we discuss the relation between the EKOR strata for different parahoric subgroups. 

\begin{proposition}\label{kk-compatible}
Let $K' \subset K$ be standard parahoric subgroups. Then for any $w \in \Adm(\{\mu\})^{K'} \cap {}^{K'} \tilde W$, 
$$\pi_{K', K} (EKOR_{K', w}) =\sqcup_{x \in \Sigma_K(w)} EKOR_{K, x}.$$
In particular, if $w \in \Adm(\{\mu\})^K \cap {}^K \tilde W$, then $\pi_{K', K} (EKOR_{K', w})=EKOR_{K, w}$.

\begin{remark}
For $w\in\Adm(\{\mu\})$, we have $W_K w W_K \subset \Adm(\{\mu\})^K$. Thus $\Sigma_K(w) \subset \Adm(\{\mu\})^K \cap {}^K \tilde W$ (the natural range of the EKOR strata for $K$). 
\end{remark}
\end{proposition}

\begin{proof}
Consider the following commutative diagram
\[\xymatrix{
 Sh_{K'}  \ar[r]^-{\Upsilon_{K'}} \ar[d]_-{\pi_{K', K}} & G(\breve \BQ_p)/\breve K'_\sigma \ar[r] \ar[d]^-{p_{K', K}} & G(\breve \BQ_p)/\breve K'_\sigma (\breve K_1 \times \breve K_1) \ar[r] \ar[d] & G(\breve \BQ_p)/\breve K'_\sigma (\breve K'_1 \times \breve K'_1) \\ 
 Sh_{K}  \ar[r]^-{\Upsilon_K} & G(\breve \BQ_p)/\breve K_\sigma \ar[r] & G(\breve \BQ_p)/\breve K_\sigma (\breve K_1 \times \breve K_1) , & }
\] where $p_{K', K}$ is the natural projection map. 

We have \begin{align*} \Upsilon_K \circ \pi_{K', K}(EKOR_{K', w}) &=p_{K', K} \circ \Upsilon_{K'} (EKOR_{K', w})=p_{K', K}(\breve K'_{\sigma} (\breve I w \breve I)/\breve K'_{\sigma}) \\ &=\breve K_{\sigma} (\breve I w \breve I)/\breve K_{\sigma}=\sqcup_{x \in \Sigma_K(w)} \breve K_{\sigma} (\breve I x \breve I)/\breve K_{\sigma} . 
\end{align*}
Here the first equality follows from the commutativity, the second equality follows from Axiom \ref{ax Y} (b), the third equality follows from the definition of $p_{K', K}$ and the last equality follows from Proposition \ref{new}. Therefore $$\pi_{K', K}(EKOR_{K', w}) \subset \Upsilon_K^{-1}(\breve K_{\sigma} (\breve I w \breve I)/\breve K_{\sigma})=\sqcup_{x \in \Sigma_K(w)} EKOR_{K, x}.$$ 

On the other hand, for any $x \in \Sigma_K(w)$ and $p \in EKOR_{K, x}$, the image of $p$ under $\Upsilon_K$ lies in $ \breve K_\sigma (\breve I x \breve I)/\breve K_\sigma \subset p_{K', K}(\breve K'_{\sigma} (\breve I w \breve I)/\breve K'_{\sigma})$. By Axiom \ref{ax Y} (b) \& (c), there exists $p' \in EKOR_{K', w}$ such that $p=\pi_{K', K}(p')$. The proposition is proved. 
\end{proof}

\

Combining Theorem \ref{nonempty-KR} with the ``in particular'' part of Proposition \ref{kk-compatible}\footnote{In fact, we only use here the inclusion $\subseteq$ in Proposition \ref{kk-compatible}. Thus Corollary \ref{6.9} does not rely on Axiom \ref{ax Y} b), c).}, we obtain the following corollary. 

\begin{corollary}\label{6.9}
Let $X_K$ be a geometric connected component of $Sh_K$. For any parahoric $K$, 
$$\upsilon_K(  X_K  )=\Adm(\{\mu\})^K \cap {}^K \tilde W.$$
{\rm In other words, any geometric connected component of $Sh_K$ intersects any EKOR stratum (as their indices run through  their natural range, i.e.,  $\Adm(\{\mu\})^K \cap {}^K \tilde W$.)}
\end{corollary}

\subsection{Closure relation}

Following \cite[\S 4]{He-07}, we introduce a partial order on ${}^K \tilde W$. 
Let $w, w' \in {}^K \tilde W$, we write $w' \preceq_{K, \sigma} w$ if there exists $x \in W_K$ such that $x w' \sigma(x)^{-1} \le w$. By \cite[4.7]{He-07}, $\preceq_{K, \sigma}$ gives a partial order on ${}^K \tilde W$. 

\begin{remark}
(1) If $w, w' \in {}^K \tilde W$ with $w' \le w$, then $w' \preceq_{K, \sigma} w$. However, the converse is not true. In other words, the partial order $\preceq_{K, \sigma}$ is a refinement of the restriction of the Bruhat order to ${}^K \tilde W$. 

(2) A more systematic view of the partial orders $\preceq_\sigma$ and $\preceq_{K, \sigma}$ is to use  minimal length elements. Let $\mathcal O, \mathcal O'$ be two $\sigma$-conjugacy classes that contain some $\sigma$-straight elements (resp. two $(W_K)_\sigma$-orbits on $\tilde W$ that contain some elements of ${}^K \tilde W$). We say that $\mathcal O' \preceq \mathcal O$ (resp. $\mathcal O' \preceq_{K, \sigma} \mathcal O$) if for some (or equivalently, any) minimal length element $w \in \mathcal O$, there exists a minimal length element $w' \in \mathcal O'$ such that $w' \le w$. These definitions coincide with the previous ones on $\sigma$-straight elements (resp. on ${}^K \tilde W$) since  $\sigma$-straight elements (resp. the elements in ${}^K \tilde W$) are minimal in their conjugacy classes (resp. their $(W_K)_\sigma$-orbits).
\end{remark}

Now we prove that the closure relation of the EKOR strata is given by this new partial order. 

\begin{theorem}\label{EKORlocclosed}
Let $K$ be a parahoric subgroup and $x \in \Adm(\{\mu\})^K \cap {}^K \tilde W$. Then $EKOR_{K, x}$ is locally closed and the closure of $EKOR_{K, x}$ is 
$$
\ov{EKOR_{K, x}}=\sqcup_{x' \in {}^K \tilde W, x' \preceq_{K, \sigma} x} EKOR_{K, x'} .
$$ 
\end{theorem}

\begin{remark}
If $x' \in {}^K \tilde W$ with $x' \preceq_{K, \sigma} x$, then by definition, there exists an element in the $(W_K)_\sigma$-orbit of $x'$ that is less than or equal to $x$ in the Bruhat order. Note that $\Adm(\{\mu\})^K$ is closed under the Bruhat order and stable under the action of $W_K \times W_K$. Thus, if $x \in \Adm(\{\mu\})^K \cap {}^K \tilde W$, then also $x' \in \Adm(\{\mu\})^K \cap {}^K \tilde W$ (the index set of EKOR strata of level $K$). 
\end{remark}

\begin{proof}
Since $\sqcup_{w \in \tilde W, w \le x} KR_{I, w}$  is the closure of $KR_{I, x}$ and the map $\pi_{I, K}$ is proper, $$\pi_{I, K}(\sqcup_{w \in \tilde W, w \le x} KR_{I, w})=\cup_{w \in \tilde W, w \le x} \pi_{I, K} (KR_{I, w}) $$ is the closure of $\pi_{I, K}(KR_{I, x})=EKOR_{K, x}$. 
We have 
\begin{align*} \cup_{w \in \tilde W, w \le x} \Upsilon_K \circ \pi_{I, K} (KR_{I, w}) & =\cup_{w \in \tilde W, w \le x}\, p_{I, K} \circ \Upsilon_I (KR_{I, w})=\cup_{w \in \tilde W, w \le x}\, p_{I, K} (\breve I w \breve I/\breve I_\sigma )\\ &=\cup_{w \in \tilde W, w \le x} \breve K_\sigma (\breve I w \breve I)/\breve K_\sigma.
\end{align*}
By \cite[Proof of Theorem 2.5]{He-11}, this equals  $\sqcup_{x' \in {}^K \tilde W, x' \preceq_{K, \sigma} x} \breve K_\sigma (\breve I x \breve I)/\breve K_\sigma$. Therefore the set 
\begin{align*} \sqcup_{x' \in {}^K \tilde W, x' \preceq_{K, \sigma} x} EKOR_{K, x'} &=\sqcup_{x' \in {}^K \tilde W, x' \preceq_{K, \sigma} x} \Upsilon_K^{-1} \big(\breve K_\sigma (\breve I w \breve I)/\breve K_\sigma\big) \\ &=\Upsilon_K^{-1} \big(\sqcup_{x' \in {}^K \tilde W, x' \preceq_{K, \sigma} x} \breve K_\sigma (\breve I x \breve I)/\breve K_\sigma\big)
\end{align*} is closed and is the closure of $EKOR_{K, x}$. 
\end{proof}

\subsection{EKOR strata and Newton strata} In this subsection, we discuss the relation between  EKOR strata and Newton strata. To do this, we need the following result which is stronger than Theorem \ref{str-bgmu} c).  

\begin{theorem}\label{strK}
For any parahoric $K$, the map 
$$
\Adm(\{\mu\})_{\sstr} \cap {}^K \tilde W \to B(G, \{\mu\})
$$
 is surjective. 
\end{theorem}

\begin{proof}
 Let $[b] \in B(G, \{\mu\})$. By Theorem \ref{str-bgmu} c), there exists $w \in \Adm(\{\mu\})_{\sstr}$ with $w \in [b]$. 

Let $J \subset \tilde{\mathbb S}$ be the set of simple reflections in $W_K$. By Proposition \ref{partial-conj}, there exists $x \in {}^J \tilde W$, an element $u \in W_{I(J, x, \sigma)}$ and an element $v \in W_K$ such that $w=v (u x) \sigma(v)^{-1}$ and $\ell(u x) \le \ell(w)$. 

Now we regard $\sigma$ as an element in the semi-direct product $\tilde W \rtimes \sigma$. By \cite[Proposition 1.2]{He-x}, there exists $N \in \mathbb N$ such that $(x \sigma)^N=(u x \sigma)^N=t^\lambda$ and $(w \sigma)^N=t^{\lambda'}$, where $\lambda, \lambda'$ are in the $W_0$-orbit of $N \nu_w$. 

Then $N \ell(x) \ge \ell((x \sigma)^N)=\langle N \nu_w, 2 \rho \rangle=N \ell(w)$. So $\ell(x) \ge \ell(w)$. On the other hand, $\ell(w) \ge \ell(u x)=\ell(u)+\ell(x)$. Therefore $u=1$ and $\ell(w)=\ell(x)$. In particular, $x$ is a $\sigma$-straight element with $x \in [b]$ and $x \in W_K w W_K \subset \Adm(\{\mu\})^K$. Since $x \in \Adm(\{\mu\})^K \cap {}^K \tilde W$, we conclude that $x \in \Adm(\{\mu\})$ by Theorem \ref{comp}. Hence $x \in \Adm(\{\mu\})_{\sstr} \cap {}^K \tilde W$.
\end{proof}

\begin{theorem}\label{EKORsubN}
For any parahoric $K$ and any $[b] \in B(G, \{\mu\})$, there exists $x \in \Adm(\{\mu\})^K \cap {}^K \tilde W$ such that $$EKOR_{K, x} \subset \mathit S_{K, [b]}.$$ 
\end{theorem}

\begin{proof}
By Theorem \ref{strK}, there exists $x \in \Adm(\{\mu\})_{\sstr} \cap {}^K \tilde W$ such that $x \in [b]$. Since $x$ is $\sigma$-straight, we deduce from  Theorem \ref{str-bgmu} a) that  $\breve K_\sigma (\breve I x \breve I) \subset [b]$. Therefore $EKOR_{K, x}=\Upsilon_K^{-1} (\breve K_\sigma (\breve I x \breve I)/\breve K) \subset \delta_K^{-1} ([b])=\mathit S_{K, [b]}$. 
\end{proof}

\begin{remarks} 1) For a general parahoric $K$, there is no KR stratum of level $K$ that is entirely contained in a given Newton stratum. 

2) For Shimura varieties of PEL type with hyperspecial level structure, the existence of an Ekedahl-Oort stratum in a given Newton stratum is proved by Viehmann/Wedhorn \cite[Theorem 1.5(1)]{V-W} and Nie \cite[Corollary 1.6]{Nie}.

\end{remarks}

\subsection{Finiteness of fibers}

\begin{proposition}\label{finite-fiber}
Let $K$ be a parahoric subgroup and $x \in {}^K \tilde W$. Then each fiber of the map $\pi_{I, K}: \breve I x \breve I/\breve I_\sigma \to \breve K_\sigma (\breve I x \breve I)/\breve K_\sigma$ is finite.
\end{proposition}

\begin{proof}
We first reformulate the statement as follows. Define an  action of $\breve I$ on $\breve K \times \breve I x \breve I$ by $i \cdot (k, z)=(k i^{-1}, i z \sigma(i)^{-1})$. Let $\breve K \times_{\breve I} \breve I x \breve I$ be its quotient. The map $\breve K \times \breve I x \breve I \to \breve K_\sigma (\breve I x \breve I)$, $(k, z) \mapsto k z \sigma(k)^{-1}$ induces a map 
\begin{equation}\label{def f}
f: \breve K \times_{\breve I} \breve I x \breve I \to \breve K_\sigma (\breve I x \breve I).
\end{equation}
The statement of the proposition is equivalent to the statement that each fiber of $f$ is finite. 

Let $J$ be the set of simple reflections in $W_K$ and $K'\subset K$ be the standard parahoric subgroup corresponding to $I(J, x, \sigma)$.  Here we are using the notation introduced right after stating Proposition  \ref{new}. Define the quotient space $\breve K \times_{\breve K'} \breve K'_\sigma (\breve I x \breve I)$ in the same way as above. By \cite[Proposition 1.10]{H-gstable}, the map $(k, z) \mapsto k z \sigma(k)^{-1}$ induces a bijection $$\breve K \times_{\breve K'} \breve K'_\sigma (\breve I x \breve I) \cong \breve K_{\sigma} (\breve I x \breve I).$$
There is a natural bijection $$\breve K \times_{\breve K'} (\breve K' \times_{\breve I} \breve I x \breve I) \cong \breve K \times_{\breve I} \breve I x \breve I.$$ Let 
\begin{equation}\label{def pi}
\pi: \breve K' \times_{\breve I} \breve I x \breve I \to \breve K'_\sigma (\breve I x \breve I)
\end{equation}
 be the map induced by $(k, z) \mapsto k z \sigma(k)^{-1}$. We have the following commutative diagram
\[\xymatrix{
\breve K \times_{\breve K'} (\breve K' \times_{\breve I} \breve I x \breve I) \ar[r]^-{(id, \pi)} \ar[d]_-{\cong} & \breve K \times_{\breve K'} \breve K'_\sigma (\breve I x \breve I) \ar[d]_-{\cong} \\ 
\breve K \times_{\breve I} \breve I x \breve I \ar[r]^-f & \breve K_\sigma (\breve I x \breve I) .}
\]
It remains to prove that each fiber of \eqref{def pi}  is finite.

Let $\breve K'_1$ be the pro-unipotent radical of $\breve K'$ and $\overline{\breve K'}=\breve K'/\breve K'_1$ be the reductive quotient of $\breve K'$. Let $\bar B$ the the image of $\breve I$ in $\overline{\breve K'}$. For any $k \in \breve K'$, we denote by $\bar k$ its image in $\overline{\breve K'}$. Note that $\Ad(x) \sigma(I(J, x,\sigma))=I(J, x, \sigma)$. Thus $\sigma':=\Ad(x) \circ \sigma$ gives a Frobenius morphism on $\overline{\breve K'}$. 

We have $\breve K'_\sigma (\breve I x \breve I) \subset \breve K' x \breve K'$.\footnote{In fact, by Lang's theorem for $\overline{\breve K'}$,  equality holds. But we do not need this fact here.} Define a map 
\begin{equation}\label{def p}
p: \breve K' x \breve K' \to \overline{\breve K'} , \quad k_1 x k_2\mapsto \bar k_1 \sigma' (\bar k_2) .
\end{equation}
 It is easy to see that this map is well-defined. Define the action of $\breve K'$ on $\overline{\breve K'}$ by $k \cdot \bar k'=\bar k \bar k' \sigma'(\bar k)^{-1}$. Then the map \eqref{def p}  is $\breve K'$-equivariant. 

The composition $p \circ \pi\colon  \breve K' \times_{\breve I} \breve I x \breve I\to \overline{\breve K'}$ is given by $(k, i x i') \mapsto \bar k \bar i (x \bar i' x^{-1})\sigma'(\bar k)^{-1}$. Note that $\bar i (x \bar i' x^{-1}) \in \bar B$. By Lang's theorem for $\bar B$, there exists $i_1 \in \breve I$ such that $\bar i_1 \bar i (x \bar i' x^{-1}) \sigma'(\bar i_1)^{-1}=1 \in \bar B$. In other words, each element in $\breve K' \times_{\breve I} \breve I x \breve I$ is represented by $(k, i x i')$ for some $k \in \breve K'$, $i, i' \in \breve I$ with $\bar i (x \bar i' x^{-1})=1$. 

Let $(k, i x i'), (k_1, i_1 x i'_1) \in \breve K' \times \breve I x \breve I$ with $\bar i (x \bar i' x^{-1})=\bar i_1 (x \bar i'_1 x^{-1})=1$. Suppose that $\pi((k, i x i'))=\pi((k_1, i_1 x i'_1))$. Then $p \pi((k, i x i'))=\bar k \sigma'(\bar k)^{-1}=\bar k_1 \sigma'(\bar k_1)^{-1}=p \pi((k_1, i_1 x i'_1))$. In other words, $\bar k_1^{-1} \bar k \in (\overline{\breve K'})^{\sigma'}$. Since $(\overline{\breve K'})^{\sigma'}$ is a finite group, each fiber of $\pi$ has only finitely many choices of $k \in \breve K'$ (up to right multiplication by $\breve I$). Therefore, each fiber of $\pi$ is a finite set.
\end{proof}

Combining Proposition \ref{finite-fiber} with the Axiom \ref{ax Y}, c)  (the finiteness part), we have 

\begin{theorem}\label{pionEKOR}
Let $K$ be a parahoric subgroup and $x \in \Adm(\{\mu\})^K \cap {}^K \tilde W$.  Then 
$${\pi_{I, K} }_{| KR_{I, x}}\colon KR_{I, x}\to EKOR_{K, x}
$$ 
is a finite morphism. In particular, $\dim EKOR_{K, x}=\dim KR_{I, x}$. 
\end{theorem}

\begin{remarks} 1) Recall that in Remark \ref{dim KR} we gave a conjectural formula for $\dim KR_{I, x}$. By Theorem \ref{pionEKOR}, this would also give a formula for the dimension of $EKOR_{K, x}$. 

2) It may be conjectured that the morphism in Theorem \ref{pionEKOR} is finite  \'etale. This would imply that all EKOR strata are smooth, which we also conjecture. This is proved by G\"ortz/Hoeve \cite{GHo} in the Siegel case.  
\end{remarks}

\section{Verification of the axioms in the Siegel case}\label{s: Siegel}

Let $g\geq 1$. Let $(V, \langle\, , \, \rangle)$ be a $\BQ$-vector space of dimension $2g$, equipped with a non-degenerate alternating form. We denote by $\mathbf G={\rm Gp}(V, \langle\, , \, \rangle)$ the group of symplectic similitudes. We fix a basis $e_1, \ldots, e_{2g}$ of $V$ such that the matrix of $\langle\, , \, \rangle$ is equal to 
\begin{equation*}
\begin{pmatrix}
0&H_g\\-H_g&0 
\end{pmatrix} ,
\end{equation*}
where $H_g$ is the unit anti-diagonal $g\times g$ matrix. 
For $j=0, \ldots, 2g-1$, we define lattices $\Lambda_j$ in $V\otimes\BQ_p$ by
\begin{equation}
\Lambda_j={\rm span}_{\BZ_p}\langle p^{-1}e_1,\ldots, p^{-1}e_j, e_{j+1},\ldots e_{2g}\rangle . 
\end{equation}
We extend this definition by periodicity to all $j\in \BZ$ by
\begin{equation*}
\Lambda_j=p^{-k}\Lambda_{\bar j},\quad j=2gk+\bar j, 0\leq \bar j\leq 2g-1 .
\end{equation*}
We consider non-empty subsets $J$ of $\BZ$ which are periodic (i.e., $J+2g\BZ=J$) and self-dual (i.e., $J=-J$). To such a subset, we associate the common stabilizer $K_J$ of the lattices $\{\Lambda_j\mid j\in J\}$ in $\mathbf G(\BQ_p)$. Then $K_J$ is a parahoric subgroup. For $J=\BZ$, we obtain an Iwahori subgroup $I$, and the map $J\mapsto K_J$ defines a bijection with all $2^g$ parahoric subgroups containing $I$. For $J=2g\BZ$ and $J=g+2g\BZ$, the parahoric subgroups $K_J$ are hyperspecial. 

We will define a moduli problem over $\Spec\, \BZ_{(p)}$ for the Shimura variety  associated to the Siegel Shimura datum $(\mathbf G, \{h\})$, where $\bK=K^p K_J$. As a preliminary, we recall some definitions from \cite[ch. 6]{RZ}.

Let $S$ be a $\BZ_{(p)}$-scheme. By a {\it $J$-set of abelian schemes} of dimension $g$ over $S$ we understand a set $A_J=\{ A_j\mid j\in J\}$ of abelian schemes of dimension $g$ over $S$ with a compatible family of isogenies
$$
\alpha_{j_1, j_2}\colon A_{j_1}\to A_{j_2}, \quad j_1<j_2
$$
of degree $p^{j_2-j_1}$ such that, for every $j\in J$, 
$$
\alpha_{j, j+2g}=p\cdot\bar\alpha_{j, j+2g} ,
$$ 
where $\bar\alpha_{j, j+2g}$ is an isomorphism. To a $J$-set of abelian schemes, we associate the {\it dual $J$-set} $\wt A_J$, defined by
$$
\wt A_j=(A_{-j})^\vee ,
$$
and where $\wt\alpha_{j_1, j_2}=(\alpha_{-j_2, -j_1})^\vee$. By a {\it principal polarization} of the $J$-set $A_J$ we mean an isomorphism of $J$-sets
$$
\lambda \colon A_J\to \wt A_J
$$
which is a polarization in the sense of \cite[Def.~6.6]{RZ}. If $0\in J$, this last condition just means that $\lambda$ induces a principal polarization $\lambda_0\colon A_0\to \wt A_0=A_0^\vee$. 

Let now $K^p\subset \mathbf G(\BA_f^p)$ be a (sufficiently small) open compact subgroup and fix $J$ as above. We consider the functor $\CM_{K^p, J}$ which to a $\BZ_{(p)}$-scheme $S$ associates the set of  isomorphism classes of triples $(A_J, \lambda, \bar\eta^p)$. Here $A_J$ is a $J$-set of abelian schemes of dimension $g$ over $S$ and $\lambda$ is a principal polarization of $A_J$. Finally, $\bar\eta^p$ is a level structure of type $K^p$ on $A_J$, i.e., a $K^p$-class of symplectic similitudes  
$$
\eta^p\colon\wh V^p(A_J)\simeq V\otimes\BA_f^p , 
$$ 
in the sense of \cite[\S 5]{Kpoints}. Here on the LHS  is the Tate module $\wh V^p(A_j)$, which is independent of $j\in J$. Then $\CM_{K^p, J}$ is representable by a quasi-projective scheme over $\Spec\,\BZ_{(p)}$ whose generic fiber is the {\it canonical model} of the Shimura variety ${\rm Sh}_{\bK}$ over $\BQ$. Furthermore, if $J=2g\BZ$, then $\CM_{K^p, J}$ is smooth over $\BZ_{(p)}$. Let us now check the axioms from section \ref{s:axioms} for the integral model ${\mathbf S}_{\bK}=\CM_{K^p, J}$.
\begin{altenumerate}
\item {\it Axiom \ref{ax funct} (compatibility with changes in the parahoric)}: It suffices to prove the desired properties for the morphism $\pi_{K_J, K_{J'}}$, when $J$ arises from  $J'$  by adding a single element $j$ (and its negative and their translates under $2g\BZ$). Let $j'\in J'$ be maximal with $j'<j$, and $j''\in J'$ be maximal with $j<j''$. Then factoring the isogeny $\alpha_{j', j''}$ as $\alpha_{j', j''}=\alpha_{j', j}\circ\alpha_{j, j''}$ is equivalent to giving a subgroup scheme of $A_{j'}[p]$ with certain properties. This is representable by a closed subscheme of a Hilbert scheme. The rest of Axiom \ref{ax funct} follows easily. 

\item {\it Axiom \ref{ax locmod} (existence of local models)}: The local model ${\mathbf M}^{\rm loc}_{K_J}$ coincides in this case with the {\it naive local model} defined in   \cite{GSp}. The morphism ${\boldsymbol \lambda}_K$ is explained and made explicit in the case at hand in \cite[\S 6]{H}.  That the strata are enumerated by $\Adm(\{\mu\})_K$ is a consequence of the equality, proved in \cite{KoR1},  between  $\Adm(\{\mu\})_K$ and ${\rm Perm}(\{\mu\})_K$ (the $\mu$-permissible set), cf. \cite[\S 4.3]{H}. The closure relation follows from the closure relation between affine Schubert cells, via the embedding of the special fiber of ${\mathbf M}^{\rm loc}_{K_J}$ into an affine flag variety  \cite{GSp}. 

\item {\it Axiom \ref{ax newton} (existence of the Newton stratification)}:  Let $x\in {Sh}_K$, and let $N$ be the common rational Dieudonn\'e module of the $J$-set of abelian varieties over $\bar\BF_p$ corresponding to $x$.  Then there exists a symplectic isomorphism 
\begin{equation}\label{compN}
N\simeq V\otimes W_\BQ(\bar\BF_p) . 
\end{equation}
The Frobenius $\mathbf F$ on $N$ can be written under the isomorphism \eqref{compN} as $\mathbf F=g\cdot(\id_V\otimes\sigma)$ with $g\in\mathbf G(W_\BQ(\bar\BF_p))$. The map $\delta_K$ now sends $x$ to the $\sigma$-conjugacy class of $g$. The compatibility with changing $K_J$ is obvious. The rest of the  axiom follows from \cite{RR}  (reduction to Grothendieck's upper semi-continuity theorem for the Newton polygon)\footnote{In fact, for the group of symplectic similitudes, this reduction is almost immediate.}. 

\item {\it Axiom \ref{ax Y} (joint stratification)}: Let $x\in Sh_K$, and let $\{ M_j\mid j\in J\}$ be the set of Dieudonn\'e modules $M(A_j)$ of the $J$-set of abelian varieties over $\bar\BF_p$ corresponding to $x$.   These form a periodic selfdual $W(\bar\BF_p)$-lattice chain inside the common {\it rational} Dieudonn\'e module $N$ of the $A_j$. By \cite[App. to ch. 3]{RZ}, there exists a symplectic isomorphism 
\begin{equation}\label{compD}
N\simeq V\otimes W_\BQ(\bar\BF_p)
\end{equation}
which carries the lattice chain $\{M_j\mid j\in J\}$ into $\{\Lambda_j\otimes W(\bar\BF_p)\mid j\in J\}$. The Frobenius $\mathbf F$ on $N$ can be written under the isomorphism \eqref{compD} as $\mathbf F=g\cdot(\id_V\otimes\sigma)$ with $g\in\mathbf G(W_\BQ(\bar\BF_p))$. The map
 $\Upsilon_K$ now sends $x$ to the class of $g$ modulo $\sigma$-conjugacy under $\breve K_J$.  The map $\lambda_K$ sends $x$ to the relative position
 $$
 \inv(M_J, \mathbf F M_J)\in W_{K_J}\backslash\tilde W/W_{K_J} , 
 $$
 hence the commutativity of the diagram appearing in Axiom \ref{ax Y} a) is clear.  
 
 We now prove part b).  It suffices to prove b) for $J=\BZ$. Indeed, we show in (v) below that $KT_{I, \tau}\neq\emptyset$, hence we may apply Lemma  \ref{prop KK'}.  
  
 Note that any principally polarized $p$-divisible group $(X, \lambda_X)$ is isomorphic to the $p$-divisible group of a principally polarized abelian variety $(A, \lambda)$. This is well-known (the proof proceeds in two steps: first, one shows that there exists a principally polarized abelian variety whose $p$-divisible group is isogenous to the given principally polarized $p$-divisible group; then one adjusts the principally polarized abelian variety in its isogeny class). Now an element of $\ell_I^{-1}(\Adm(\{\mu\}))$ mapping to the Dieudonn\'e module of $(X, \lambda_X)$ corresponds to a maximal isotropic chain of subgroups $G_\bullet$ of $X[p]$. Under the isomorphism of $X[p]$ with $A[p]$, the chain $G_\bullet$ defines a maximal isotropic chain of subgroups of $A[p]$, and hence a principally polarized $\BZ$-set of abelian varieties mapping to the chosen point of $\ell_I^{-1}(\Adm(\{\mu\}))$. 
 
 We now prove part c). It suffices to prove c) for the pair $K=K_J, K'=K_{J'}$, when $J$ arises from  $J'$  by adding a single element $j$ (and its negative and their translates under $2g\BZ$). Let $j'\in J'$ be maximal with $j'<j$, and $j''\in J'$ be maximal with $j<j''$. Let $x'$ be a point of $Sh_{K'}$ and let $y' $ be its image under $\Upsilon_{K'}$. Then $y'$ corresponds to the polarized $J'$-chain of principally polarized $p$-divisible groups $(X_\bullet, \lambda_{X,\bullet})$. As remarked in (i) above, and taking part b) into account,  giving a point $y\in \ell_K^{-1}(\Adm(\{\mu\})_K)$ above $y'$  is equivalent to giving a subgroup scheme $G$ of $X_{j'}[p]$ with certain properties.

The set of $\ov{\BF}_p$-points of the leaf above  $y'$  is given by  
\begin{equation}\label{leaf hyper}
\{ (A_\bullet, \lambda_\bullet, \bar\eta^p, \alpha)\mid \alpha\colon (A[p^\infty]_\bullet, \lambda_\bullet)\to (X_\bullet, \lambda_{X, \bullet})\}/\Aut(X_\bullet, \lambda_{X, \bullet}) .
\end{equation}
Here $\alpha$ denotes an isomorphism of polarized $J'$-sets of $p$-divisible groups over $\ov{\BF}_p$. 

The set of $\ov{\BF}_p$-points of the leaf above $y$  is given by  
\begin{equation}
\{ (A_\bullet, \lambda_\bullet, \bar\eta^p, \alpha)\mid \alpha\colon (A[p^\infty]_\bullet, \lambda_\bullet)\to (X_\bullet, \lambda_{X, \bullet})\}/\Aut(X_\bullet, \lambda_{X, \bullet}, G) .
\end{equation}
Here $\Aut(X_\bullet, \lambda_{X, \bullet}, G)$ denotes the subgroup of $\Aut(X_\bullet, \lambda_{X, \bullet})$ consisting of those automorphisms that preserve $G$. Indeed, to $(A_\bullet, \lambda_\bullet, \bar\eta^p, \alpha)$, where the first entry is a polarized $J'$-set of abelian varieties, we associate  the $J$-set of abelian varieties $A$  defined by $(A_\bullet, \lambda_\bullet, \bar\eta^p)$ and $\alpha^{-1}(G)$ (a  finite flat group scheme contained in $A_{j'}[p]$). 

The quotient $\Aut(X_\bullet, \lambda_{X, \bullet})/\Aut(X_\bullet, \lambda_{X, \bullet}, G)$ is finite, since $\Aut(X_\bullet, \lambda_{X, \bullet}, G)$ contains the subgroup of $\Aut(X_\bullet, \lambda_{X, \bullet})$ of automorphisms which induce the identity automorphism on $X_{j'}[p],$ which obviously has finite index in $\Aut(X_\bullet, \lambda_{X, \bullet})$. It follows that the fibers of the  morphism between the leaves above $y$, resp. above $y'$, are all finite of the same cardinality.

 \item {\it Axiom \ref{ax nonemptytau} (basic non-emptiness)}:  
 We first show that $KR_{I, \tau}$ is non-empty. 
 
 Consider the super-special principally polarized abelian variety $(A_0, \lambda_0)$ over $\bar\BF_p$, which arises as the $g$th power of a supersingular elliptic curve $E$, with its canonical principal polarization. Let $H$ be the kernel of the Frobenius endomorphism of $E$. Then $(A_0, \lambda_0)$, with its  totally isotropic complete flag by finite flat group schemes $(0)\subset G_1\subset\ldots\subset G_g \subset A[p]$ given by 
 $$
 G_i=H\times\ldots H \times (0)\ldots\times (0)\subset E^i\times E^{g-i} ,
 $$
 and with a suitable  level structure $\bar\eta^p$ defines a point  of $KR_{I, \tau}$.
 
 To deduce Axiom \ref{ax nonemptytau}, we appeal to Remark \ref{rem comp}. For the Siegel case a theory of compactifications exists, hence we deduce an identification
 \begin{equation}\label{pi-reduction}
 \pi_0(\CM_{K^p,\, \BZ}\otimes \ov{\BF}_p)=\pi_0({\rm Sh}_{\bK}) ,
 \end{equation}
 where $\bK=K^p\cdot I$ (here $I$ denotes the Iwahori subgroup corresponding to $J=\BZ$). The RHS of \eqref{pi-reduction} can be identified with
 $$
 \pi_0({\rm Sh}_{\bK})=\BA_f^\times/\BQ_+^\times\cdot c(\bK) ,
 $$
 where  $c\colon \bG\to \BG_m$ denotes the multiplier morphism, cf. \cite{De}. It follows that $\bG(\BA_f^p)$ acts transitively on $ \varprojlim_{K^p}\pi_0({\rm Sh}_{\bK})$. In other words, by changing the prime-to-$p$ level structure in $(A, \lambda, \bar\eta^p)$, one can pass to any connected component of ${ Sh}_I$.
   By varying the level structure, we obtain a point of $KR_{I, \tau}$ in an arbitrary connected component of ${\rm Sh}_I$. 
\end{altenumerate}

\end{document}